\documentclass[12pt]{amsart}

\usepackage{amsthm, amsfonts, amssymb, xypic}
\usepackage[all]{xy}
\usepackage{amscd}
\usepackage[pagebackref,colorlinks]{hyperref}

\theoremstyle{definition}
\newtheorem{ntn}{Notation}[section]

\theoremstyle{plain}
\newtheorem{lem}[ntn]{Lemma}

\newtheorem{prp}[ntn]{Proposition}
\newtheorem{thm}[ntn]{Theorem}
\newtheorem{cor}[ntn]{Corollary}

\theoremstyle{remark}

\newtheorem{rem}[ntn]{Remark}
\newtheorem{exa}[ntn]{Example}

\newcommand{\z}{\mathbb{Z}}

\newcommand{\R}{\mathbb{R}}

\newcommand{\lan}{\langle}
\newcommand{\ran}{\rangle}

\newcommand{\GL}{\mathit{{\rm GL}}}

\newcommand{\SL}{\mathit{{\rm SL}}}

\newcommand{\inc}{{\rm inc}}
\newcommand{\id}{{\rm id}}

\newcommand{\tors}{{{\rm Tor}_1^{\z}}}

\newcommand{\si}{\sigma}

\newcommand{\arr}{\rightarrow}
\newcommand{\larr}{\longrightarrow}

\newcommand{\mt}{\mapsto}

\newcommand{\fff}{{R^\ast}}

\newcommand{\rr}{{R^{\ast}}}

\newcommand{\diag}{{\rm diag}}
\renewcommand{\ker}{{\rm ker}}
\newcommand{\coker}{{\rm coker}}
\newcommand{\im}{{\rm im}}
\newcommand{\ind}{{\rm ind}}

\newtheoremstyle{athm}
  {}
  {}
  {\itshape}
  {}
  {\scshape}
  {}
  {.5em}
  {\thmnote{#3}}
\theoremstyle{athm}

\begin{document}

\title[Third homology of general linear groups]
{Third homology of general linear groups over rings with many units}
\author{Behrooz Mirzaii}
\begin{abstract}
For a commutative ring $R$ with many units, we describe the kernel of
$H_3(\inc): H_3(\GL_2(R), \z) \arr H_3(\GL_3(R), \z)$.
Moreover we show that the elements of this kernel are of order at
most two. As an application we study the indecomposable part of $K_3(R)$.
\end{abstract}
\maketitle

\section*{Introduction}

Interest in the study of the homology of general linear groups
has arose  mostly because of their close connection with the $K$-theory of rings.
For any ring $R$ and any positive integer $n$, there are natural homomorphisms
\[
\begin{array}{ccc}
K_n(R) \!\!\!& \overset{h_n}{-\!\!\!-\!\!\!-\!\!\!-\!\!\!-\!\!\!-\!\!\!-\!\!\!\larr}
& \!\!\!\!\! H_n(\GL(R), \z)\\
h_n'\searrow\!\!\!\!\!\!\!\!\!\!\!\!\!\!\!\!\! & &
\!\!\!\!\!\!\!\hspace{-2cm}\swarrow\!\!\!\nearrow\\
& H_n(E(R), \z), &
 \end{array}
\]
where $E(R)$ is the elementary subgroup of the stable general linear group $\GL(R)$
and $h_n$ and $h_n'$ ($n \ge 2$ for $h_n'$) are
the Hurewicz maps coming from algebraic topology
\cite[Chap. 2]{srinivas1996}.

It is known that $K_1(R)\overset{h_1}{\simeq} H_1(\GL(R), \z)$,
$K_2(R)\overset{h_2'}{\simeq} H_2(E(R), \z)$
\cite[Chap. 2]{srinivas1996}. The homomorphism $h_3':K_3(R) \arr H_3(E(R), \z)$
is surjective with 2-torsion kernel
\cite[Corollary 5.2]{suslin1991}, \cite[Proposition 2.5]{sah1989}.

Homological stability type theorems, are very powerful tools for the study
of $K$-theory of rings. Suslin has proved that for an infinite field $F$,
we have the homological stability
\[
H_n(\GL_n(F), \z) \overset{\sim}{\larr}  H_n(\GL_{n+1}(F), \z) \overset{\sim}{\larr}
H_n(\GL_{n+1}(F), \z) \overset{\sim}{\larr} \cdots,
\]
and used this to prove many interesting results \cite{suslin1985}.
For example he showed that we have an exact sequence
\[
\begin{CD}
H_n(\GL_{n-1}(F), \z)
@>{H_n(\inc)}>>
H_n(\GL_n(F), \z) \larr
K_n^M(F) \larr 0.
\end{CD}
\]
Suslin has conjectured that the kernel of
\[
H_n(\GL_{n-1}(F), \z) \arr  H_n(\GL_n(F), \z)
\]
is a torsion  group \cite[Problem 4.13]{sah1989}.
These results can be generalized over rings with many units \cite{guin1989}, e.g.
semilocal rings with infinite residue fields. Also Suslin's conjecture can be asked
in this more general setting \cite{mirzaii2008}. A positive answer to this conjecture only is
known for $n \le 4$ \cite{elbaz1998}, \cite{mirzaii-2008}, \cite{mirzaii2008}.

It was known that when $F$ is an infinite field, the kernel of the homomorphism
$H_3(\GL_2(F), \z) \arr H_3(\GL_3(F), \z)$ is a 2-power torsion group \cite{mirzaii-2008}.
In this article we generalize this to all commutative rings with many units.
In fact we do more. Here we describe the kernel of
\[
H_3(\inc): H_3(\GL_2(R), \z) \arr H_3(\GL_3(R), \z),
\]
where $R$ is a commutative ring with many units.
Our main theorem claims that the  elements of $\ker(H_3(\inc))$ are of the form
\[
{\begin{array}{c} \!\!\sum \!\!\end{array}}
{\rm \bf{c}} (\diag(a,1), \diag(1,b),\diag(c, c^{-1}))
\]
provided that
\[
{\begin{array}{c} \!\!\sum \!\!\end{array}}
a\otimes \{b, c\}+b\otimes \{a, c\}=0 \in \fff \otimes_\z K_2^M(R).
\]
Moreover by an easy argument we will show that $\ker(H_3(\inc))$ is a $2$-torsion
group. It is highly expected that this kernel should be trivial, at least when $R$ 
is a field \cite[Section 5]{hutchinson-tao2009}.

It is known that, the map $H_3(\inc)$ is closely related to  the indecomposable part
of $K_3(R)$, i.e. $K_3(R)^\ind:=K_3(R)/K_3^M(R)$ \cite{mirzaii-2008},
\cite{hutchinson-tao2009}. As an application of our main theorem we show that
\[
K_3(R)^\ind \otimes_\z \z[1/2] \simeq H_0(\fff, H_3(\SL_2(R), \z[1/2])).
\]
If $\fff=\fff^2$, then we get the isomorphism
\[
K_3(R)^\ind \simeq H_3(\SL_2(R), \z).
\]
Previously these results were only known
for infinite fields \cite{mirzaii-2008}.
\subsection*{Notation}
In this article by $H_i(G)$ we mean  the  homology of group $G$
with integral coefficients, namely $H_i(G, \z)$.
By $\GL_n$ (resp. $\SL_n$) we mean the general (resp. special) linear
group $\GL_n(R)$ (resp. $\SL_n(R)$), where  $R$ is a commutative ring with $1$.
If $A \arr A'$ is a homomorphism of abelian groups, by $A'/A$ we
mean $\coker(A \arr A')$ and we take other liberties of this kind.
For a group $A$, by $A_{\z[1/2]}$ we mean $A\otimes_\z \z[1/2]$.

\section{Third homology of product of two abelian groups}

In this section we will study the homology group $H_3(A\times B)$, where $A$ and $B$ are
abelian groups.

First we assume $A=B=\z/n$. By applying the  K\"unneth formula
\cite[Proposition 6.1.13]{weibel1994} to $H_3(\z/n \times \z/n)$
and using the calculation of the homology of
finite cyclic groups \cite[Theorem 6.2.2, Example 6.2.3]{weibel1994},
we obtain the exact sequence
\[
0 \arr H_3(\z/n)\oplus H_3(\z/n) \arr H_3(\z/n \times \z/n)
\arr \tors(\z/n, \z/n) \arr 0.
\]
If $p_i: \z/n \times \z/n \larr \z/n$, $i=1,2$, is projection on the $i$-th factor, then
\[
({p_1}_\ast, {p_2}_\ast): H_3(\z/n \times \z/n) \larr H_3(\z/n)\oplus H_3(\z/n)
\]
splits the above exact sequence. Thus we obtain a canonical splitting map
\[
\theta_{n,n}: \tors(\z/n, \z/n) \larr H_3(\z/n \times \z/n).
\]
If $\lan \bar{1}, n, \bar{1} \ran$ is the image of $\bar{1}\in  \z/n$ under the
isomorphism
\[
\z/n \overset{\simeq}{\larr} \tors(\z/n, \z/n),
\]
then one can show that $\theta_{n,n}(\lan \bar{1}, n, \bar{1} \ran)=\chi_{n,n}$, where
\begin{gather*}
\chi_{n,n}:= {\begin{array}{c} \!\! \sum \!\!\end{array}}_{i=1}^{n}
\Big([(\bar{1},0)|(0,\bar{1})|(0,\bar{i})]
-[(0,\bar{1})|(\bar{1},0)|(0,\bar{i})]\\
\hspace{2.3 cm}
+ [(0,\bar{1})|(0,\bar{i})|(\bar{1},0)]
+[(\bar{1},0)|(\bar{i},0)|(0,\bar{1})]\\
\hspace{2.6 cm}
 - [(\bar{1},0)|(0,\bar{1})|(\bar{i},0)]
+[(0,\bar{1})|(\bar{1},0)|(\bar{i},0)]\Big).
\end{gather*}
\cite[Chap. V, Proposition 10.6]{maclane1963}, \cite[Proposition 4.1]{mirzaii-2008}.
If $A=\z/m$ and $B=\z/n$, then the same approach shows that the exact sequence
\[
0 \!\arr\! H_3(\z/m)\oplus H_3(\z/n) \!\arr \!H_3(\z/m \times \z/n)
\!\arr \!\tors(\z/m, \z/n)\! \arr \!0,
\]
splits canonically. The splitting map
\[
\theta_{m,n}: \tors(\z/m, \z/n) \larr H_3(\z/m \times \z/n)
\]
can be computed similar to  $\theta_{n,n}$. In fact if
$\lan \overline{m/d}, d, \overline{n/d} \ran$ is the image of $\bar{1} \in \z/(m,n)$
under the isomorphism $\z/(m,n) \simeq \tors(\z/m, \z/n)$, then
$\theta_{m,n}(\lan \overline{m/d}, d, \overline{n/d} \ran)=\chi_{m,n}$, where

\begin{gather*}
\begin{array}{rl}
\vspace{2mm}
\chi_{n,m}:= & \!\!\!\! \sum_{i=1}^{n}
\Big([(\overline{\frac{m}{d}},0)|(0,\overline{\frac{n}{d}})|(0,\overline{\frac{in}{d}})]
-[(0,\overline{\frac{n}{d}})|(\overline{\frac{m}{d}},0)|(0,\overline{\frac{in}{d}})]\\
\vspace{2mm}
&\hspace{0.6cm}
+ [(0,\overline{\frac{n}{d}})|(0,\overline{\frac{in}{d}})|(\overline{\frac{m}{d}},0)]
+[(\overline{\frac{m}{d}},0)|(\overline{\frac{im}{d}},0)|(0,\overline{\frac{n}{d}})]\\
\vspace{2mm}
&\hspace{0.6cm}
 - [(\overline{\frac{m}{d}},0)|(0,\overline{\frac{n}{d}})|(\overline{\frac{im}{d}},0)]
+[(0,\overline{\frac{n}{d}})|(\overline{\frac{m}{d}},0)|(\overline{\frac{im}{d}},0)]\Big).
\end{array}
\end{gather*}

In the next proposition we extend these results to all abelian groups.

\begin{prp}\label{canonical}
Let $A$ and $B$ be abelian groups. Then we have the canonical decomposition
\[
H_3(A \times B)= \bigoplus_{i+j=3} H_i(A) \otimes H_j(B)
\oplus \tors(A,B).
\]
\end{prp}
\begin{proof}
By the K\"unneth formula we have the exact sequence
\[
0 \larr \bigoplus_{i+j=3} H_i(A)\otimes H_j(B) \larr H_3(A \times B)
\larr \tors(A,B) \larr 0.
\]
We will construct a canonical splitting map
\[
\tors(A,B) \larr H_3(A \times B).
\]
It is known that direct limit with directed set index, is an exact functor and it
commutes with the homology group \cite[Chap. V, Section 5, Exercise 3]{brown1994}
and the functor Tor \cite[Corollary 2.6.17]{weibel1994}.
Since any abelian group can be written as direct limit of its finitely
generated subgroups, we may assume that $A$ and $B$
are finitely generated abelian groups. On the other hand,
\[
\tors(A,B) \simeq \tors(A_{tor}, B_{tor}),
\]
where $A_{tor}$ is the subgroup of torsion elements of $A$.
So we may even assume that $A$ and $B$ are finite abelian groups. Let
\[
A=\z/m_1 \times \dots \times \z/m_r, \ \ \
B=\z/n_1 \times \dots \times \z/n_s.
\]
Now consider the commutative diagram
\[
\begin{CD}
0 \!\arr &\! H_3(\z/m_i)\!\oplus \!H_3(\z/n_j)& \!\!\!\arr &
\!H_3(\z/m_i \!\times \!\z/n_j)&\! \arr & \!\tors(\z/m_i,\! \z/n_j)& \arr \!0\\
 & @VVV @VV{\inc_{m_i,n_j}}V @VV{\inc}V & \\
 \!\!0 \!\arr& \!\bigoplus_{i+j=3} H_i(A)\otimes H_j(B) &
 -\!\!\!-  & \!\!\!\!\!\!\!\!\!\!\!\arr H_3(A \times B)&
-\!\!\!-& \!\!\!\!\!\!\!\!\!\!\!\arr\tors(A,B)& \!\!\!\!\!\!\!\larr\! 0.
\end{CD}
\]
We have seen that the first row of this diagram splits by the canonical map $\theta_{m_i,n_j}$.
Thus
the composition
\[
\begin{CD}
\tors(\z/m_i, \z/n_j) @>{\inc_{m_i,n_j}\circ \theta_{m_i,n_j}}>> H_3(A \times B) \larr \tors(A,B)
\end{CD}
\]
is the natural inclusion map. Since
\[
\tors(A,B)=\bigoplus_{\underset{1\le j\le s}{1\le i\le r}}\tors(\z/m_i, \z/n_j),
\]
we obtain a map $\theta_{A,B}: \tors(A,B) \larr H_3(A \times B)$
that decomposes our exact sequence canonically. In fact
$\theta_{A,B}=\sum_{i,j} \inc_{m_i,n_j} \circ \theta_{m_i,n_j}$.
\end{proof}
\section{The third homology of $\GL_2$}

A commutative ring $R$ with $1$ is called a {\it ring with many units}
if for any $n \ge 2$ and for any finite number of surjective linear
forms $f_i: R^n \arr R$, there exists a $v \in  R^n$ such that, for
all $i$, $f_i(v) \in \fff$. Important examples of rings with many units
are semilocal rings with infinite residue fields. In particular for an
infinite  field $F$, any commutative finite dimensional $F$-algebra
is a semilocal ring and so is a ring with many units.
In this article we always assume that $R$ is a commutative ring with many units.

Let
\[
\fff^3 \times \GL_0  \overset{\inc}{\hookrightarrow}
\fff^2 \times \GL_1 \overset{\inc}{\hookrightarrow} \fff \times \GL_2
\overset{\inc}{\hookrightarrow} \GL_3
\]
be the natural diagonal inclusions. Here by $\fff^n$ we mean
$\fff \times \cdots \times \fff$ ($n$-times). Let
\begin{gather*}
\begin{array}{ll}
\sigma_2^1:=\inc:\fff \times \GL_2 \larr \GL_3, &\\
&\\
\vspace{1.5mm}
\sigma_1^1:\fff^2 \times \GL_1 \larr \fff \times \GL_2,&
(a,b,c) \mt (b,a,c),\\
\sigma_1^2=\inc: \fff^2 \times \GL_1 \larr \fff \times \GL_2,&
(a,b,c) \mt (a,b,c),\\
&\\
\vspace{1.5mm}
\sigma_0^1:\fff^3 \times \GL_0  \larr \fff^2 \times \GL_1,& (a,b,c) \mt (b,c,a),\\
\vspace{1.5mm}
\sigma_0^2:\fff^3 \times\GL_0  \larr \fff^2 \times\GL_1,&(a,b,c) \mt (a, c, b),\\
\sigma_0^3=\inc:\fff^3 \times \GL_0  \larr \fff^2 \times\GL_1,&
(a,b,c) \mt (a,b,c).
\end{array}
\end{gather*}

It is easy to see that the chain of maps
\[
\begin{CD}
H_3(\fff^3 \times \GL_0)  @>{\sigma_0^1}_\ast-{\sigma_0^2}_\ast+
{\sigma_0^3}_\ast>> H_3(\fff^2 \times \GL_1) \\
@>{\sigma_1^1}_\ast-{\sigma_1^2}_\ast>>
\!\!\!\!\!\!\!\!\!\!\!\!\!H_3(\fff \times \GL_2)\!\!\!\!\!\!\!\!\!\!\!\!\!
@>{\sigma_2^1}_\ast>>\! H_3(\GL_3) \!\larr 0
\end{CD}
\]
is a chain complex. The following result has been proved in
\cite[Corollary 3.5]{mirzaii-2008}.

\begin{thm}\label{exact-GL3}
The sequence
\[
\begin{CD}
H_3(\fff^2 \times \GL_1) @>{\sigma_1^1}_\ast-{\sigma_1^2}_\ast>>
H_3(\fff \times \GL_2) @>{\sigma_2^1}_\ast>> H_3(\GL_3) \larr 0
\end{CD}
\]
is exact.
\end{thm}



Using the K\"unneth formula \cite[Proposition 6.1.13]{weibel1994}, we
have the decomposition
$H_3(\fff \times \GL_{2})= {\begin{array}{c} \!\! \bigoplus \!\!\end{array}}_{i=0}^4 S_i$,
where
\begin{gather*}
\begin{array}{l}
\vspace{1.5mm}
S_0=H_3(\GL_2),\\
\vspace{1.5mm}
S_i=H_i(\fff) \otimes H_{3-i}(\GL_{2}),\ \ 1 \le i \le 3, \\
\vspace{1.5mm}
S_4=\tors(\rr, H_1(\GL_2))\simeq \tors(\mu(R), \mu(R)).
\end{array}
\end{gather*}
Note that by the homological stability, $\fff \simeq H_1(\GL_1) \simeq H_1(\GL_2)$
\cite[Theorem 1]{guin1989}. This decomposition is canonical.
The splitting map
\[
S_4 \simeq \tors(\mu(R), \mu(R)) \larr H_3(\fff \times \GL_{2})
\]
is given by the composition
\begin{gather*}
S_4 \simeq \tors(\mu(R),\mu(R))
 \overset{\theta_{R,R}}{\larr}  H_3(\fff \times \fff)
\overset{q_\ast}{\larr}  H_3(\fff \times \GL_{2}),
\end{gather*}
where
\[
q:\rr \times \fff \larr \fff \times \GL_{2}, \ \ (a, b) \mt (a, b, 1),
\]
and $\theta_{R,R}$ is obtained from Proposition \ref{canonical}. Using the decomposition
\begin{gather*}
H_2(\GL_2)=H_2(\GL_1) \oplus K_2^M(R)
\end{gather*}
\cite[Theorem 2]{guin1989}, we have $S_1=S_1' \oplus S_1''$, where
\begin{gather*}
S_1'=\rr \otimes H_2(\GL_1),\ \ \
S_1''=\rr \otimes K_2^M(R).
\end{gather*}
We should remark that the  inclusion $K_2^M(R) \larr H_2(\GL_2)$, in
the decomposition of $H_2(\GL_2)$, is given by the formula
\[
\{a,b\} \mt {\rm \bf{c}} (\diag(a,1),\diag(b,b^{-1}))
\]
\cite[Proposition A.11]{elbaz1998}.
For the definition of Milnor's $K$-groups, $K_n^M(R)$, over commutative rings and
their study over rings with many units, we refer the interested readers to
subsection 3.2 of \cite{guin1989}.

Let us to introduce the notation ${\rm \bf{c}} (-,-)$ in a more general setting
and state some of its main properties. These will be used frequently in
this article. Let $G$ be a group and set
\[
{\rm \bf{c}}(g_1, g_2,\dots, g_n):=\sum_{\si \in \Sigma_n} {{\rm
sign}(\si)}[g_{\si(1)}| g_{\si(2)}|\dots|g_{\si(n)}] \in
H_n(G),
\]
where $g_1, \dots, g_n \in G$ pairwise commute and $\Sigma_n$ is the symmetric
group of degree $n$. Here we use the bar resolution of $G$
\cite[Chapter I, Section 5]{brown1994} to define the homology of $G$.

\begin{lem}
Let $G$ and $G'$ be two groups.
\par {\rm (i)} If $h_1\in G$ commutes with all the elements
$g_1, \dots, g_n \in G$, then
\[
{\rm \bf{c}}(g_1h_1, g_2,\dots, g_n)= {\rm \bf{c}}(g_1, g_2,\dots,
g_n)+{\rm \bf{c}}(h_1, g_2,\dots, g_n).
\]
\par {\rm (ii)}
For every $\sigma \in \Sigma_n$, ${\rm \bf{c}}(g_{\sigma(1)},\dots,
g_{\sigma(n)})={\rm sign(\sigma)} {\rm \bf{c}}(g_1,\dots,
g_n)$.
\par {\rm (iii)}
The cup product of ${\rm \bf{c}}(g_1,\dots, g_p)\in H_p(G)$
and ${\rm \bf{c}}(g_1',\dots, g_q') \in H_q(G')$ is ${\rm
\bf{c}}((g_1, 1), \dots, (g_p,1),(1,g_1'), \dots, (1,g_q')) \in
H_{p+q}(G \times G')$.
\end{lem}
\begin{proof}
The proofs follow from direct computations, so we leave it to the interested readers.
\end{proof}

Again using the K\"unneth formula and Proposition \ref{canonical}, we obtain the canonical
decomposition  $H_3(\fff^2 \times \GL_1)=\bigoplus_{i=0}^8 T_i$, where

\begin{gather*}
\begin{array}{ll}
\vspace{1.5mm}
T_0=H_3(\GL_1),\\
\vspace{1.5mm}
T_1=\bigoplus_{i=1}^3 T_{1,i}=\bigoplus_{i=1}^3 H_i(R_1^\ast)\otimes H_{3-i}(\GL_1),\\
\vspace{1.5mm}
T_2=\bigoplus_{i=1}^3 T_{2,i}=
\bigoplus_{i=1}^3 H_i(R_2^\ast)\otimes H_{3-i}(\GL_1),\\
\vspace{1.5mm}
T_3=R_1^\ast \otimes R_2^\ast\otimes H_1(\GL_1),\\
\vspace{1.5mm}
T_4=\tors(R_1^\ast, R_2^\ast)\simeq \tors(\mu(R), \mu(R)),\\
\vspace{1.5mm}
T_5= \tors(R_1^\ast, H_1(\GL_1))\simeq \tors(\mu(R), \mu(R)), \\
\vspace{1.5mm}
T_6=\tors(R_2^\ast, H_1(\GL_1))\simeq \tors(\mu(R), \mu(R)),\\
\vspace{1.5mm}
T_7=R_1^\ast\otimes H_2(R_2^\ast),\\
\vspace{1.5mm}
T_8=H_2(R_1^\ast)\otimes R_2^\ast.
\end{array}
\end{gather*}
Here by $R_i^\ast$ we mean the $i$-th component of
$\fff \times \cdots \times\fff$.
Now we give an explicit description of restriction of the map
$\alpha:={\sigma_1^1}_\ast-{\sigma_1^2}_\ast$
on all $T_i$'s. By direct computations one sees that
\[
\begin{array}{ll}
\vspace{1.5mm}
\alpha|_{T_0}: T_0 \larr S_0, & x \mt 0, \\
\vspace{1.5mm}
\alpha|_{T_{1,i}}:T_{1,i} \larr S_0 \oplus S_i, &
x_i \otimes x_i' \mt (x_i \cup x_i', -x_i \otimes x_i'), 1 \le i \le 3,\\
\vspace{1.5mm}
\alpha|_{T_{2,i}}:T_{2,i} \larr S_0 \oplus S_i, &
y_i \otimes y_i' \mt (-y_i \cup y_i', y_i \otimes y_i'), 1 \le i \le 3,\\
\vspace{1.5mm}
\alpha|_{T_3}: T_3 \larr S_1,
& a\otimes b \otimes c \mt -b \otimes (a\cup c)-a \otimes(b\cup c),\\
\vspace{1.5mm}
\alpha|_{T_4}: T_4 \larr S_4, & z \mt 0,\\
\vspace{1.5mm}
\alpha|_{T_5}: T_5 \larr S_0 \oplus S_4, & u \mt ({\sigma_1^1}_\ast(u), -u),\\
\vspace{1.5mm}
\alpha|_{T_6}: T_6 \larr S_0 \oplus S_4, & v \mt (-{\sigma_1^2}_\ast(v), v),\\
\vspace{1.5mm}
\alpha|_{T_7}: T_7 \larr S_1 \oplus S_2, &
d\otimes u'\mt (-d\otimes u', u' \otimes d),\\
\vspace{1.5mm}
\alpha|_{T_8}: T_8 \larr S_1 \oplus S_2, & v'\otimes e
\mt (e\otimes v', -v'\otimes e),
\end{array}
\]
where $x\cup y$ is the cup product of $x$ and $y$.
\section{The kernel of $H_3(\GL_2)\arr H_3(\GL_3)$}

Our goal in this article is to study the kernel of the map
$\inc_\ast:H_3(\GL_2)\arr H_3(\GL_3)$.
So let $x\in \ker(\inc_\ast)$. Then
\[
(x, 0, 0, 0, 0) \in \ker ({\sigma_2^1}_\ast)
\subseteq \bigoplus_{i=0}^4 S_i=H_3(\fff \times \GL_2).
\]

By Theorem \ref{exact-GL3} and by the explicit description
of $\alpha={\sigma_1^1}_\ast-{\sigma_1^2}_\ast$ given in the previous section,
there exists an element
\[
l=(0,(x_i\otimes x_i')_{1 \le i \le 3},
(y_i\otimes y_i')_{1 \le i \le 3}, {\begin{array}{c} \!\! \sum \!\! \end{array}}
 a\otimes b \otimes c, 0, u,v, d\otimes u',
v' \otimes e)
\]
in $H_3(\fff^2 \times \GL_1)$ such that $\alpha(l)=(x, 0, 0, 0, 0)$.

Set $\beta:={\sigma_0^1}_\ast-{\sigma_0^2}_\ast+{\sigma_0^3}_\ast$, and
consider the following summands of $H_3(\fff^3 \times \GL_0)$,
\[
T_1':=R_1^\ast \otimes H_2(R_2^\ast), \ \ \ \ \ T_2':=H_2(R_1^\ast) \otimes R_2^\ast.
\]
By easy computations one sees that
\[
\begin{array}{ll}
\vspace{1.5mm}
\beta|_{T_1'}:T_1' \larr T_{1,1} \oplus T_{1,2} \oplus T_7,&
f\otimes w \mt (-f\otimes w, w\otimes f, f\otimes w)\\
\beta|_{T_2'}:T_2' \larr T_{1,1} \oplus T_{1,2} \oplus T_8,&
w'\otimes f' \mt (f'\otimes w', -w'\otimes f', w'\otimes f').
\end{array}
\]
So we may assume $d\otimes u'=0, v' \otimes e=0$. Therefore we have
\[
\begin{array}{l}
\vspace{1.5mm}
\sum_{i=1}^3 x_i \cup x_i' - \sum_{i=1}^3 y_i \cup y_i' +
{\sigma_1^1}_\ast(u) -{\sigma_1^2}_\ast(v)=x,\\
\vspace{1.5mm}
-x_1 \otimes x_1' +  y_1 \otimes y_1' -
\sum [b \otimes (a\cup c)+a \otimes(b\cup c)]=0,\\
\vspace{1.5mm}
-x_2 \otimes x_2' +  y_2 \otimes y_2'=0,\\
\vspace{1.5mm}
-x_3 \otimes x_3' +  y_3 \otimes y_3'=0,\\
-u+v=0.
\end{array}
\]
Therefore we obtain the following relations
\[
\begin{array}{l}
\vspace{1.5mm}
x=x_1 \cup x_1' -  y_1 \cup y_1' \in S_0=H_3(\GL_2),\\
x_1 \otimes x_1' -  y_1 \otimes y_1' =
-\sum b \otimes (a\cup c)+a \otimes(b\cup c) \in S_1.
\end{array}
\]
Under the decomposition $H_2(\GL_2)=H_2(\GL_1)\oplus K_2^M(R)$, we have
\[
a\cup b= {\rm \bf{c}} (\diag(a,1),\diag(1, b))=({\rm \bf{c}} (a, b) , \{a,b\}).
\]
Thus under the decomposition $S_1=S_1' \oplus S_1''$, we have
\[
\Big(x_1 \! \otimes x_1'\!\! - \! y_1 \!\otimes y_1'\! +
\!\!{\begin{array}{c}\!\! \sum \!\! \end{array}} b\otimes {\rm \bf{c}} (a, c)\!+\!
a\otimes {\rm \bf{c}} (b, c), \! {\begin{array}{c} \!\!\sum \!\! \end{array}}
 b\otimes \!\{a, c\}
\! + a\otimes \! \{b, c\} \!\Big)\! =\! 0,
\]
and hence
\begin{gather*}
x_1 \otimes x_1' - y_1 \otimes y_1'=-{\begin{array}{c}\!\! \sum \!\! \end{array}}
 b\otimes {\rm \bf{c}} (a, c)+
a\otimes {\rm \bf{c}} (b, c),\\
{\begin{array}{c}\!\! \sum \!\! \end{array}} b\otimes \{a, c\}+ a\otimes \{b, c\}=0.
\end{gather*}
Therefore
\[
\begin{array}{rl}
\vspace{2mm}
x \!\!\!\! & = -\sum {\rm \bf{c}} (\diag(a,1), \diag(1,b),\diag(1,c))\\
\vspace{1.5mm}
& \ \ \ \ \ \ \ \ \! + {\rm \bf{c}} (\diag(b,1), \diag(1,a),\diag(1,c))\\
&= \sum {\rm \bf{c}} (\diag(a,1), \diag(1,b),\diag(c,c^{-1})),
\end{array}
\]
such that
${\begin{array}{c} \!\! \sum \!\! \end{array}}a\otimes \{b, c\}+ b\otimes \{a, c\}=0$.
From now on, we will use the following notation:
\[
l_{a,b,c}
= {\rm \bf{c}} (\diag(a,1), \diag(1,b),\diag(c, c^{-1}))
\]

Hence we have proved most parts of the following theorem.

\begin{thm}\label{kernel}
Let $R$ be a commutative ring with many units. Then the kernel of
$\inc_\ast:H_3(\GL_2) \arr H_3(\GL_3)$  consists of elements of the form
${\begin{array}{c} \!\! \sum \!\! \end{array}}
{\rm \bf{c}} (\diag(a,1), \diag(1,b),\diag(c, c^{-1}))$
provided that
\[
{\begin{array}{c} \!\! \sum \!\! \end{array}}
a\otimes \{b, c\}+b\otimes \{a, c\}=0 \in \fff \otimes K_2^M(R).
\]
In particular $\ker(\inc_\ast) \subseteq \fff\cup H_2(\GL_1) \subseteq H_3(\GL_2)$,
where the cup product is induced by the diagonal inclusion
$\inc: \fff \times \GL_1 \arr \GL_2$. Moreover $\ker(\inc_\ast)$ is a $2$-torsion
group.
\end{thm}
\begin{proof}
The only part that remains to be proved is that $\ker(\inc_\ast)$ is a 2-torsion
group. Let $x \in \ker(\inc_\ast)$. For simplicity we may assume that
$x=l_{a,b,c}={\rm \bf{c}} (\diag(a,1), \diag(1,b),\diag(c, c^{-1}))$,
such that $a\otimes \{b, c\}+b\otimes \{a, c\}=0$.
Let $\Phi$ be the following composition
\[
\fff \otimes K_2^M(R) \overset{\id_\fff \otimes \iota}{\larr}
\fff \otimes H_2(\GL_2) \overset{\cup}{\larr}
 H_3(\fff \times \GL_2) \overset{\alpha_\ast}{\larr} H_3(\GL_2),
\]
where $\iota:K_2^M(R) \arr H_2(\GL_2)$ is described in the previous section,
$\cup$ is the cup product
and $\alpha: \fff \times \GL_2 \arr \GL_2$  is given by $(a, A) \mt aA$.
It is easy to see that
\[
\Phi(a \otimes \{b, c\})={\rm \bf{c}} (\diag(a,a), \diag(b,1),\diag(c,c^{-1})).
\]
Now with easy computations, one sees that
\[
\begin{array}{rl}
\vspace{1.5mm}
0\!\!\!\!&=\Phi(0)\\
\vspace{1.5mm}
&=\Phi(a\otimes \{b, c\}+b\otimes \{a, c\})\\
\vspace{1.5mm}
&={\rm \bf{c}} (\diag(a,a), \diag(b,1),\diag(c,c^{-1}))\\
\vspace{1.5mm}
&\ +{\rm \bf{c}} (\diag(b,b), \diag(a,1),\diag(c,c^{-1}))\\
&=-2l_{a,b,c}.
\end{array}
\]
This completes the proof of the theorem.
\end{proof}

\begin{rem}
One can show directly that if
$a\otimes \{b, c\}+b\otimes \{a, c\}=0$, then
$l_{a,b,c} \in \ker(\inc_\ast:H_3(\GL_2)\arr H_3(\GL_3))$. To see this, let $\Psi$
be the following composition
\[
\fff \otimes K_2^M(R) \overset{\id_\fff \otimes \iota}{\larr}
\fff \otimes H_2(\GL_2) \overset{\cup}{\larr}
 H_3(\fff \times \GL_2)
 {\larr} H_3(\GL_3).
\]
Then it is easy to see that
\[
\Psi(a \otimes \{b,c\})={\rm \bf{c}} (\diag(a,1,1), \diag(1,b,1),\diag(1,c,c^{-1})).
\]
Now we have
\[
\begin{array}{rl}
\vspace{1.5mm}
\inc_\ast(l_{a,b,c})\!\!\!\!& =+{\rm \bf{c}} (\diag(1, a, 1),
\diag(1,1,b),\diag(1,c,c^{-1}))\\
\vspace{1.5mm}
&  =+{\rm \bf{c}} (\diag(a,1,1), \diag(1,b,1),\diag(c,c^{-1},1))\\
\vspace{1.5mm}
&= -{\rm \bf{c}} (\diag( a,1,1), \diag(1,b,1),\diag(1,c,1)\\
\vspace{1.5mm}
& \ \ \  -{\rm \bf{c}} (\diag(b, 1,1), \diag(1,a,1),\diag(1,c,1)\\
\vspace{1.5mm}
&= -{\rm \bf{c}} (\diag(a,1,1), \diag(1,b,1),\diag(1,1,c,))\\
\vspace{1.5mm}
&\ \ \  -{\rm \bf{c}} (\diag(a,1,1), \diag(1,b,1),\diag(1, c,c^{-1}))\\
\vspace{1.5mm}
&\ \ \  -{\rm \bf{c}} (\diag(b,1,1), \diag(1,a,1),\diag(1,1,c))\\
\vspace{1.5mm}
&\ \ \  -{\rm \bf{c}} (\diag(b,1,1), \diag(1,a,1),\diag(1,c,c^{-1}))\\
\vspace{1.5mm}
&=-\Psi(a\otimes \{b, c\}+b\otimes \{a, c\})\\
&=0.
\end{array}
\]
\end{rem}
~
\\
\begin{cor}\label{injectivity}
Let $R$ be a ring with many units.
\par { \rm (i)} The natural map $\inc_\ast: H_3(\GL_2, \z[1/2])\arr H_3(\GL_3, \z[1/2])$
is injective.
\par { \rm (ii)} If $\fff=\fff^2=\{a^2|a \in \fff\}$, then
$\inc_\ast: H_3(\GL_2)\arr H_3(\GL_3)$ is injective.
\end{cor}
\begin{proof}
The part (i) immediately follows from Theorem \ref{kernel}.
Let $\fff=\fff^2$. By Theorem \ref{kernel}, we may assume that
$x \in \ker(\inc_\ast)$ is of the form
$l_{a,b,c} \in H_3(\GL_2)$ such that $a\otimes \{b, c\}+b\otimes \{a, c\}=0$. Let
$c={c'}^2$ for some $c' \in \fff$. Then
$l_{a,b,c}=2l_{a,b,c'}$ and $2\Big(a\otimes \{b, c'\}+b\otimes \{a, c'\}\Big)=0$.
Since $K_2^M(R)$ is uniquely $2$-divisible \cite[Proposition 1.2]{bass-tate1973},
$\fff \otimes K_2^M(R)$ is uniquely
$2$-divisible too. Hence $a\otimes \{b, c'\}+b\otimes \{a, c'\}=0$. Now from
Theorem \ref{kernel}, it follows that $2l_{a,b,c'}=0$.
Therefore $l_{a,b,c}=0$ and hence
$\inc_\ast:H_3(\GL_2) \arr H_3(\GL_3)$ is injective.
\end{proof}

\begin{exa}
Let $R=\R$.
It is well-know that $K_2^M(\R)\simeq \lan \{-1,-1\}\ran \oplus V$,
where $V$ is uniquely divisible and is generated by elements $\{a, b\}$
with $a,b > 0$. Let $l_{a,b,c} \in H_3(\GL_2(\R))$ such that
$a\otimes \{b, c\}+b\otimes \{a, c\}=0$.
If $a>0$, then
$a\otimes \{b,c\}= a \otimes \{-b,c\}=a\otimes \{b,-c\}=a\otimes \{-b,-c\}$,
so we may assume that $b,c>0$. Now with an argument as in the proof of  the previous
corollary, one sees that $l_{a,b,c}=0$. A similar argument works if $b>0$ or if $c>0$.
If $a, b, c <0$, then one can easily reduce the problem to the case
that $a=b=c=-1$, and it is trivial to see that $l_{-1,-1,-1}=0$.
Therefore $\inc_\ast:H_3(\GL_2(\R)) \arr H_3(\GL_3(\R))$ is injective.
\end{exa}

\begin{rem}
Consider the following  chain of maps
\[
\fff^{\otimes 3} \otimes K_0^M(R) \!\overset{\delta_0^{(3)}}{\larr}
\!\fff^{\otimes 2} \otimes K_1^M(R)\! \overset{\delta_1^{(3)}}{\larr}
\!\fff \otimes K_2^M(R)\! \overset{\delta_2^{(3)}}{\larr} \!K_3^M(R)\! \arr\! 0,
\]
where
\[
\begin{array}{l}
\vspace{1.5mm}
\delta_2^{(3)}: a \otimes \{b,c\} \mt \{a,b,c\}\\
\vspace{1.5mm}
\delta_1^{(3)}: a \otimes b \otimes \{c\} \mt a \otimes \{b,c\}+b \otimes \{a,c\}\\
\vspace{1.5mm}
\delta_0^{(3)}: a \otimes b \otimes c \mt
b \otimes c \otimes \{a\}+ a \otimes c \otimes \{b\}+ a \otimes b \otimes \{c\}.
\end{array}
\]
It is easy to see that this is, in fact, a chain complex.
It is not difficult to see that $\ker(\delta_2^{(3)})=\im(\delta_1^{(3)})$
(see the proof of Theorem 3.2 in \cite{hutchinson-tao2009}).
Under the composition
\[
\fff^{\otimes 3} \larr \fff \otimes H_2(\fff) \larr H_3(\GL_2),
\]
defined by
\[
 a \otimes b \otimes c \mt a \otimes {\rm \bf{c}} (b,c) \mt
 {\rm \bf{c}} (\diag(a,1), \diag(1,b),\diag(1,c)),
\]
one can see that $\im(\delta_0^{(3)})$ maps to zero.
Thus we obtain a surjective map
\[
\ker(\delta_1^{(3)})/\im(\delta_0^{(3)}) \larr
\ker\Big(H_3(\GL_2) \arr H_3(\GL_3)\Big),
\]
\[
{\begin{array}{c} \!\! \sum \!\!\end{array}}
 a \otimes b \otimes c + \im(\delta_0^{(3)})\mt
 {\begin{array}{c} \!\! \sum \!\!\end{array}} l_{a,b,c}.
\]

\end{rem}


\begin{lem}\label{inj-SL2}
Let $R$ be a ring with many units.
\par { \rm (i)} We have the exact sequence
\[
0\! \larr \! H_3(\SL_2, \z[1/2])_\fff \!\larr \!H_3(\SL, \z[1/2])
\!\larr \!K_3^M(R)_{\z[1/2]}\! \larr \!0.
\]
\par { \rm (ii)} If $\fff=\fff^2=\{a^2|a \in \fff\}$, then we have the exact sequence
\[
0 \larr  H_3(\SL_2) \larr H_3(\SL)
\larr K_3^M(R) \larr 0.
\]
\end{lem}
\begin{proof}
The proof is similar to the proof of Theorem 6.1 and Corollary 6.2
in \cite{mirzaii-2008}.
\end{proof}

\begin{thm}\label{ind}
Let $R$ be a ring with many units.
\par { \rm (i)} We have the isomorphism
\[
K_3(R)^\ind \otimes \z[1/2]\simeq H_3(\SL_2, \z[1/2])_\fff.
\]
\par { \rm (ii)} If $\fff=\fff^2=\{a^2|a \in \fff\}$, then
\[
K_3(R)^\ind\simeq  H_3(\SL_2).
\]
\end{thm}
\begin{proof}
The proof is similar to the proof of Theorem 6.4 in \cite{mirzaii-2008}.
\end{proof}

\begin{rem}
Previously Lemma \ref{inj-SL2} and Theorem \ref{ind} were only known for
infinite fields \cite[Corollary 6.2, Proposition 6.4]{mirzaii-2008}.
\end{rem}

\subsection*{Acknowledgments}
Part of this work has done during my visit to ICTP on August 2011.
I would like to thank them for their support and hospitality.


\bigskip
\address{{\footnotesize

Department of Mathematics,

Institute for Advanced Studies in Basic Sciences,

P. O. Box. 45195-1159, Zanjan, Iran.

email:\ bmirzaii@iasbs.ac.ir

}}
\end{document}